\documentclass[runningheads]{llncs}
\usepackage{amsfonts}
\usepackage{amsmath}
\usepackage{amssymb}
\usepackage{tikz}
\usepackage{tikz}
\usepackage{longtable}
\usepackage{array}
\newcolumntype{x}[1]{>{\centering\arraybackslash\hspace{0pt}}p{#1}}
\usepackage[font=normalsize,labelfont={bf}]{caption}
\spnewtheorem{coro}[theorem]{Corollary}{\bfseries}{\itshape}

\newcommand{\zed}{{\ensuremath{\mathbb{Z}}}}

\newcommand{\eff}{{\ensuremath{\mathbb{F}}}}

\begin{document}

\title{Orthogonal and strong frame starters, revisited\thanks{D.R.\ Stinson's research is supported by  NSERC discovery grant RGPIN-03882.
}
\author{Douglas R.\ Stinson\inst{1}}}
\institute{David R.\ Cheriton School of Computer Science\\ University of Waterloo\\
Waterloo, Ontario, N2L 3G1, Canada \\ \email{dstinson@uwaterloo.ca}}
\authorrunning{D.R.\ Stinson}

\maketitle

\begin{abstract}
In this paper, I survey frame starters, as well as orthogonal and strong frame starters, in abelian groups. I mainly recall and re-examine existence and nonexistence results, but I will prove some new results as well. 

\end{abstract}

\section{Introduction}

At the \emph{Stinson '66} conference, Jeff Dinitz gave a very interesting and inspiring talk entitled
``Jeff and Doug's Excellent Math Adventure.''  In his talk, Jeff discussed several unresolved problems that arose in various joint papers we have written over the past 40+ years. There are many interesting research areas that have long-standing, interesting open problems, but which perhaps have been not so active in recent years. In this paper, I decided to revisit a research topic that was of considerable interest in the 1970's and 1980's, namely orthogonal and strong frame starters. I will mainly survey and re-examine existence and nonexistence results, but I will prove some new results as well. 

The rest of this paper is organized as follows.
Section \ref{fs.sec} discusses frame starters. Section \ref{strong.sec} addresses strong and orthogonal frame starters. In each section, I review past work and prove some new results. I should mention that strong starters give rise to Room squares and strong and orthogonal frame starters give rise to Room frames. I do not discuss these objects in this paper; however, information about orthogonal and strong starters  
can be found \cite{DS92a,GL75,Ho89,KSV85,LO10}. Section \ref{cyclic.sec} looks at strong frame starters in cyclic groups. Here I recall an old, but still unsolved, conjecture concerning the existence of  strong  starters in arbitrary abelian groups. I also propose a new conjecture relating to the existence of strong frame starters in cyclic groups. Section \ref{discuss.sec} is a brief discussion and conclusion to the paper.

\section{Frame Starters}
\label{fs.sec}

I begin with some definitions.

\begin{definition}
Let $G$ be an additive abelian group of order $g$ and let $H$ be a subgroup of $G$ of order $h$. A \emph{frame starter} in $G \setminus H$ is a set of $(g-h)/2$ pairs $\{ \{x_i,y_i \} : 1 \leq i \leq (g-h)/2\}$ that satisfies the following two properties:
\begin{enumerate}
\item $\{ x_i, y_i :  1 \leq i \leq (g-h)/2 \} = G \setminus H$.
\item $\{ \pm (x_i-y_i) :  1 \leq i \leq (g-h)/2 \} = G \setminus H$.
\end{enumerate}
This frame starter has \emph{type} $h^{g/h}$.
\end{definition}
Note that the pairs in the frame starter form a partition of $G \setminus H$, and the differences obtained from these pairs also partitions $G \setminus H$.

I should remark that I am always writing group operations additively in this paper. Also, although I am not considering nonabelian groups, I should mention that there are a few results pertaining to nonabelian groups that have been shown (see, for example, \cite{Gr76,Wang98}).

Here are three very elementary observations concerning frame starters.

\begin{lemma}
\label{elem.lem}
Let $G$ be an abelian group of order $g$ and let $H$ be a subgroup of $G$ of order $h$.
If there is a frame starter $S$ in $G \setminus H$, then the following hold:\vspace{-.1in}
\begin{enumerate}
\item $g-h$ is even.
\item $G \setminus H$ contains no elements of order $2$.
\item $g \geq 4h$.
\end{enumerate}
\end{lemma}

\begin{proof} 
The first statement is obvious, since the pairs in $S$ partition $G \setminus H$.

The second statement is proven as follows. Suppose $d \in G \setminus H$ has order $2$.
There must be a pair $\{x,y\} \in S$ such that $x - y = d$. But then we also have $y - x = d$, so the difference $d$ occurs twice. This is not allowed.

Finally, we prove that $g \geq 4h$. Of course $h \mid g$, so we need to rule out $g=2h$ and $g=3h$. 

If $g = 2h$, then there are two cosets of $H$ in $G$, namely, $H$ and $H + z$ for some $z \not\in H$. If $\{x,y\}$ is a pair in $S$, then $x,y \in H+z$. But then $x-y \in H$, which is not allowed.  

If $g = 3h$, then there are three cosets of $H$ in $G$, namely, $H$, $H + z$ and $H+z'$ for some $z,z' \not\in H$. If $\{x,y\}$ is a pair in $S$, then WLOG $x \in H+z$ and $y \in H+z'$. But then $x+y \in H$, which is not allowed. \qed 
\end{proof}
 
When $H = \{0\}$, a frame starter in $G \setminus H$ is just called a \emph{starter}. A starter can only exist in a group of odd order. Starters have received much study in the past.
In this paper, I mainly consider frame starters where $|H| > 1$.

It is still not known exactly which groups and subgroups admit frame starters. But here is one interesting general existence result originally due to Rosa \cite{Rosa66}. Rosa proved the result in the setting of Skolem sequences; see Rees and Stinson \cite{RS92} for additional discussion. This theorem was later re-discovered by Wang \cite{Wang95}.

\begin{theorem}
There is a frame starter in $\zed_{2n} \setminus \{0,n\}$ for all $n \equiv 0,1 \bmod 4$.
\end{theorem}

A frame starter $S$ in $G \setminus H$ is \emph{patterned} if $S = \{ \{x,-x\}: x \in G \setminus H\}$. The following theorem is  due to Wang.

\begin{theorem}
\cite{Wang98}
$G \setminus H$ has a patterned frame starter if and only if all elements of $G \setminus H$ have odd order.
\end{theorem}

\begin{coro}
$G \setminus H$ has a patterned frame starter if $G$ has odd order.
\end{coro}

The following non-existence result was proven by Anderson \cite{An-80} in the special case where $G$ is a cyclic group. It was observed in \cite{DS80} without proof that a similar result holds in any abelian group. I include the proof of the more general result here for completeness.

\begin{theorem}
\label{nonexist2.thm}
\cite{An-80,DS80}
Suppose $G$ is an abelian group of order $2u$ and suppose $H$ is a subgroup of $G$ of order $2t$, where $t$ is odd. If  $u/t \equiv 2\text{ or }3 \bmod 4$, then there is no frame starter in $G \setminus H$.
\end{theorem}

\begin{proof}
Suppose first that $u/t \equiv 3 \bmod 4$. Then $G \cong \zed_2 \times G'$, where $|G'| = u$ is odd, and $H \cong \zed_2 \times  H'$, where $|H'| = t$ is odd and $H'$ is a subgroup of $G'$. 
Define the homomorphism $\phi : G \rightarrow \zed_2 $ by $\phi(i,g') = i$.

Suppose $S = \{ \{x_i,y_i \} : 1 \leq i \leq u-t\}$ is a frame starter in $G \setminus H$. 
An element $g \in G$ is \emph{even} or \emph{odd} according to whether $\phi(g) = 0$ or $1$. 
An unordered pair $\{x_i,y_i \}$ is \emph{even} or \emph{odd} according to whether $\phi(x_i-y_i) = 0$ or $1$. 
Note that $S$ must contain $(u-t)/2$ even pairs and  $(u-t)/2$ odd pairs.

The \emph{type} of a pair $\{x_i,y_i \}$ is defined to be $\{\phi(x_i), \phi(y_i)\}$.
For $T = \{0,0\}, \{0,1\}$ or $\{1,1\}$, let $n_T$ denote the number of pairs in $S$ of type $T$.
Counting even and odd elements in the pairs in $S$, we have 
\[2 n_{\{1,1\}} + n_{\{0,1\}} = 2 n_{\{0,0\}} + n_{\{0,1\}} = u-t.\] 
Further, the odd differences only occur in pairs of type $\{0,1\}$, so 
\[ n_{\{0,1\}} = \frac{u-t}{2}.\]
It then follows that
\[ n_{\{1,1\}} = n_{\{0,0\}} = \frac{u-t}{4}.\]
Now $u = st$, where $s \equiv 3 \bmod 4$, so $u-t = (s-1)t \equiv 2 \bmod 4$ since $t$ is odd. 
Thus $(u-t)/4$ is not an integer, which is a contradiction.

The proof  when $u/t \equiv 2 \bmod 4$ is based on similar techniques. But first we have to consider the structure of $G$. Here we have $G \cong \zed_4 \times G'$ or $G \cong \zed_2 \times \zed_2 \times G'$, where $|G'|$ is odd. Also,
$H \cong \zed_2 \times H'$, where $|H'|$ is odd. 
In order for a frame starter to exist, there cannot be any elements of order $2$ in $G \setminus H$. Thus we must have $G \cong \zed_4 \times G'$ and $H = \{0,2\} \times H'$. 

Define the homomorphism $\phi : G \rightarrow \zed_2 $ by $\phi(i,g') = i \bmod 2$.
We again define an element in $g \in $ to be \emph{even} or \emph{odd} according to whether $\phi(g) = 0$ or $1$, and an unordered pair $\{x_i,y_i \}$ is \emph{even} or \emph{odd} according to whether $\phi(x_i-y_i) = 0$ or $1$. 

Note that all $2t$ elements of $H$ are even, so 
$G \setminus H$ contains $u-2t$ even elements and $u$ odd elements. 
Also, $S$ must contain $(u-2t)/2$ even pairs and $u/2$ odd pairs.

Counting even and odd elements in the pairs in $S$, 
we have \[2 n_{\{0,0\}} + n_{\{0,1\}} = u-2t\]
and
 \[2 n_{\{1,1\}} + n_{\{0,1\}}  = u.\]
The odd differences only occur in pairs of type $\{0,1\}$, so 
\[ n_{\{0,1\}} = \frac{u}{2}.\]
However, it then follows that
\[ n_{\{1,1\}} = \frac{u}{4} \quad \text{and} \quad n_{\{0,0\}} = \frac{u}{4}-t,\]
which is is impossible since these values must be integers.
\qed
\end{proof}

\section{Strong and Orthogonal Frame Starters}
\label{strong.sec}

I now turn  to strong and orthogonal frame starters. Here are the relevant definitions.

\begin{definition}
Suppose that $S_1 = \{ \{x_i,y_i \} : 1 \leq i \leq (g-h)/2\}$ and $S_2 = \{ \{u_i,v_i \} : 1 \leq i \leq (g-h)/2\}$
are both frame starters in $G \setminus H$. Without loss of generality, assume that
$y_i - x_i = v_i - u_i$ for $1 \leq i \leq (g-h)/2$.  $S_1$ and $S_2$ are
\emph{orthogonal} if the following two properties hold:
\begin{enumerate}
\item $y_i - v_i \not\in H$ for $1 \leq i \leq (g-h)/2$.
\item $y_i - v_i \neq y_j - v_j$ if $1 \leq i,j \leq (g-h)/2$, $i \neq j$. 
\end{enumerate}
\end{definition}
In other words, when the pairs in $S_1$ and $S_2$ are matched according to their differences, the ``translates'' are distinct elements of $G \setminus H$. These translates are often called the \emph{adder}.

\begin{definition}
Suppose that $S = \{ \{x_i,y_i \} : 1 \leq i \leq (g-h)/2\}$ is a frame starter in $G \setminus H$. $S$  is \emph{strong}
 if the following two properties hold:
\begin{enumerate}
\item $x_i + y_i \not\in H$ for $1 \leq i \leq (g-h)/2$.
\item $x_i + y_i  \neq x_j + y_j$ if $1 \leq i,j \leq (g-h)/2$, $i \neq j$. 
\end{enumerate}
\end{definition}

It is easy to see that a frame starter $S$ is strong if and only if $S$ is orthogonal to $-S$.

\begin{example}
\label{zed10}
Suppose $G =\zed_{10}$ and $H = \{0,5\}$. Here is a strong frame starter of type $2^5$ in $G\setminus H$:
\[
\begin{array}{l}
S = \{ \{3,4\}, \{7,9\}, \{8,1\}, \{2,6\} \}. 
\end{array}
\]
\end{example}

\begin{example}
Suppose $G =\zed_{7}$ and $H = \{0\}$. Here is a strong frame starter  of type $1^7$ in $G\setminus H$:
\[
\begin{array}{l}
S = \{ \{2,3\}, \{5,1\}, \{6,4\} \}. 
\end{array}
\]
\end{example}

\begin{example}
\cite{DS80}
\label{z15.exam}
Suppose $G =\zed_{15}$ and $H = \{0,5,10\}$. Here are two orthogonal frame starters  of type $3^5$ 
in $G\setminus H$:
\[
\begin{array}{l}
S_1 = \{ \{1,2\}, \{9,11\}, \{3,6\}, \{8,12\}, \{13,4\}, \{7,14\} \} \\
S_2 = \{ \{2,3\}, \{11,13\}, \{9,12\} , \{4,8\}, \{1,7\}, \{14,6\}\}. 
\end{array}
\]
It is easy to compute the adder associated with these starters:
\[
\begin{array}{c|c|c}
S_1 & \multicolumn{1}{|c|}{\text{adder}} & S_2\\ \hline
\{1,2\} & 1 & \{2,3\} \\
\{9,11\} & 2 &  \{11,13\}\\
\{3,6\} & 6 & \{9,12\}\\
\{8,12\} & 11 &  \{4,8\}\\
\{13,4\} & 3 &  \{1,7\}\\
\{7,14\} & 7 &  \{14,6\}
\end{array}
\]
\end{example}

\begin{example}
\cite{St80}
\label{z4z4.exam}
Suppose $G =\zed_{4} \times \zed_{4}$ and $H = \{(0,0), (0,2), (2,0), (2,2)\}$. Here is  a strong frame starter of type $4^4$ in $G\setminus H$:
\[
\begin{array}{ll}
S = &\{ \{(1,1),(3,2)\}, \{(3,0),(3,1)\}, \{(2,1),(3,3)\}, \{(0,3),(1,3)\},\\
&   \{(1,0),(2,3)\}, \{(0,1),(1,2)\} \}. 
\end{array}
\]
\end{example}

When $H = \{0\}$,  orthogonal and strong frame starters in $G \setminus H$ are  called  \emph{orthogonal starters}
and \emph{strong starters}, resp. These have received considerable study, but, as was the case with starters, I mainly consider  orthogonal and strong frame starters with $|H| > 1$ in this paper.

Orthogonal frame starters in $G \setminus H$, of type $h^n$, where $|H| = h$ and $|G| = nh$, can be used to construct a \emph{uniform frame} of type $h^n$. When $h = 1$, the frame is known as a \emph{Room square}.
This frame has $G$ in its automorphism group. For information about uniform frames and Room squares, see 
\cite{DL93,DS92a,DSZ94,DW10,GZ93,MSVW81} . 

Here are two known families of strong frame starters.

\begin{theorem}\cite{DS80,SW80}
If $q \equiv 1 \bmod 4$ is a prime power and $t \geq 1$, then there is a strong frame starter in
$(\eff_q \times (\zed_2)^t) \setminus (\{0\} \times (\zed_2)^t)$.
\end{theorem}

\begin{theorem}\cite{AG78}
If $p \equiv 1 \bmod 6$ is a prime, $p \geq 19$, then there is a strong frame starter in
$(\zed_p \times \zed_3) \setminus (\{0\} \times \zed_3)$.
\end{theorem}

There are various recursive constructions for strong frame starters using strong complete mappings (also known as ``strong orthomophisms''), which I define now. 
Let $G$ be an (additive) abelian group. A mapping $\phi : G \rightarrow G$ is a \emph{strong complete mapping} if the following conditions are satisfied:
\begin{enumerate}
\item $\phi$ is a bijection,
\item the mapping $x \mapsto \phi(x) - x$ is a bijection,
\item the mapping $x \mapsto \phi(x) + x$ is a bijection.
\end{enumerate}

Finite abelian groups that admit strong complete mappings have been completely characterized by Evans \cite{Ev12,Ev13}.
Evans proved that a finite abelian group admits a strong complete mapping if and only if neither its Sylow $2$-subgroup nor its Sylow $3$-subgroup is nontrivial and cyclic.
 
Our next construction is a simple multiplication construction. This construction was first stated in this form in \cite{St81};  however, similar constructions were given earlier in \cite{AG77,Gr74}.

\begin{theorem}
\cite{St81}
\label{mult1.thm}
Suppose $G_1$, $G_2$ are abelian groups and $H$ is a subgroup of $G_1$. Suppose there is a strong frame starter in $G_1 \setminus H$ and suppose  there is a strong complete mapping in
$G_2$. Then there is a strong frame starter in $(G_1 \times G_2) \setminus (H \times G_2)$.
\end{theorem}

\begin{proof}
Let $S$ be a strong frame starter in $G_1 \setminus H$. $S$ consists of $(g-h)/2$ pairs, where $|G_1| = g$ and $|H| = h$. Let $\phi$ be a strong complete mapping of $G_2$. For every pair $\{x,y\} \in S$, order the pair arbitrarily, obtaining $(x,y)$. Then replace $(x,y)$ by $|G_2|$ ordered pairs, namely $((x,z),(x,\phi(z))$, for all $z \in G_2$. Finally replace all these ordered pairs by unordered pairs. The result is a strong frame starter in $(G_1 \times G_2) \setminus (H \times G_2)$.\qed
\end{proof}

It is easy to see that, if $q \geq 4$ is a prime power, then the mapping $x \mapsto \alpha x$ is a strong complete mapping in $\eff_q$ provided that $\alpha \neq 0, 1$ or $-1$. So the following corollary of 
Theorem \ref{mult1.thm} is obtained.

\begin{coro}
\label{mult1.cor}
Suppose there is a strong frame starter in $G\setminus H$ and suppose $q \geq 4$ is a prime power. Then there is a strong frame starter in $(G \times \eff_q) \setminus (H \times \eff_q)$.
\end{coro}

\begin{coro}
\label{5h.coro}
If $q \geq 4$ is a prime power, then there is a strong frame starter in $(\zed_{10} \times \eff_q) \setminus (\{0,5\} \times \eff_q)$.
\end{coro}

\begin{proof}
There is a strong frame starter in $\zed_{10} \setminus \{0,5\}$ (see Example \ref{zed10}). Apply   Corollary \ref{mult1.cor} with $G = \zed_{10}$ and $H = \{0,5\}$.
\end{proof}

\begin{remark}\label{5h.rem} 
Observe that the constructed strong frame starter in Corollary \ref{5h.coro} has type $h^5$, where $h = 2q$. Also, 
$\zed_{10} \times \eff_q \cong \zed_{10q}$ if $q > 5$ is an odd prime, so we obtain a strong frame starter
in $\zed_{10q} \setminus \{0,5,10, \dots , 10q-5\}$ for such a value of $q$. For example, if $q=7$, the result is a strong frame starter (of type $14^5$) in $\zed_{70} \backslash \{0,5,10, \dots , 65\}$.
\end{remark}

The following variation of Theorem \ref{mult1.thm} is a straightforward generalization of \cite[Proposition 7]{Wang98}.

\begin{theorem}
\label{mult1-var.thm}
Suppose $G_1$, $G_2$ are abelian groups and $H$ is a subgroup of $G_1$. Suppose there is a strong frame starter in $G_1 \setminus H$ and a strong frame starter in $(H \times G_2) \setminus (H \times \{0\})$. Further, suppose  there is a strong complete mapping in
$G_2$. Then there is a strong frame starter in $(G_1 \times G_2) \setminus (H \times \{0\})$.
\end{theorem}

\begin{proof}
First use Theorem \ref{mult1.thm} to construct a strong frame starter in 
$(G_1 \times G_2) \setminus (H \times G_2)$. Then adjoin the pairs in a strong frame starter in $(H \times G_2) \setminus (H \times \{0\})$. The desired strong frame starter results.\qed
\end{proof} 

Wang uses the Theorem \ref{mult1-var.thm} to prove the following result.

\begin{theorem}\cite{Wang98}
Suppose $n$ is a positive integer, all of whose prime factors are congruent to $1$ modulo $4$. Then
there is a strong frame starter in $\zed_{2n} \setminus \{0,n\}$.
\end{theorem}

In 1974, Gross \cite{Gr74} proved the following multiplication theorem. 

\begin{theorem}\cite{Gr74}
\label{gr.thm}
If there exist strong starters in $G_1$ and $G_2$, where $|G_1|$ is odd and $\gcd(|G_2|,6) = 1$, then there is a strong starter in $G_1 \times G_2$.
\end{theorem}

Theorem \ref{gr.thm} can be obtained as a corollary of Theorem \ref{mult1-var.thm}, as follows.
Suppose we take $H = \{0\}$ in Theorem \ref{mult1-var.thm} and we also assume that $\gcd(|G_2|, 6) = 1$. Then there is a strong complete mapping in $G_2$. Applying Theorem \ref{mult1-var.thm}, we obtain Theorem \ref{gr.thm}.

I should also mention the following recursive construction due to Horton \cite{Ho71}. As far as I am aware, an analog of this result for frame starters has not been proven.

\begin{theorem}\cite{Ho71}
If there is a strong starter in $G$, where $\gcd(|G|,6) = 1$, then there is a strong starter in 
$\zed_5 \times G$.
\end{theorem}

The previous results concerned strong starters in the direct product of two groups. 
The next theorem is closely related to a theorem of Gross and Leonard \cite{GL75}. It permits the construction of a strong starter in an abelian group $G$, given appropriate structures in a subgroup $H$ of $G$ and the quotient group $G / H$. 

\begin{theorem}
\label{GL.thm}
Suppose $G$ is an abelian group and $H$ is a subgroup of $G$. Suppose there is a strong  starter in the quotient group $G / H$ and suppose there is a strong complete mapping in
$H$. Then there is a strong frame starter in $G \setminus H$.
\end{theorem}

\begin{proof}
Suppose $|G|= g$ and $|H| = h$ and denote $n = g/h$. Let the elements of $H$ be  $z_1, \dots , z_h$. There is a strong starter $S$ in $G/H$, which consists of $(n-1)/2$ pairs, say $\{H + x_i,H+y_i\}$ for $1 \leq i \leq (n-1)/2$. Replace each pair $\{H + x_i,H+y_i\}$ in $S$ by an arbitrarily ordered pair $(H + x_i,H+y_i)$.
Then replace each such ordered pair by $h$ pairs $(z_j + x_i, \phi(z_j) + y_i)$, $1 \leq j \leq h$. Finally, replace each such ordered pair by an unordered pair. The $h(n-1)/2 = (g-h)/2$ resulting pairs comprise a strong frame starter in $G \setminus H$.\qed
\end{proof}

\begin{remark} In the result proven by Gross and Leonard \cite{GL75}, they assume (in addition to the hypotheses of Theorem \ref{GL.thm}) that there is a strong starter in $H$. They then conclude that there is a strong starter in $G$. In Theorem \ref{GL.thm}, we just omit the strong starter in $H$, and the result is the stated strong frame starter in $G \setminus H$. 
\end{remark}

\subsection{Nonexistence Results}

I now discuss  nonexistence results.  I review some previously published results and give a couple of new ones, using a unified approach. The proofs make use of the canonical homomorphism from $G$ to $G/H$. 
Suppose  $S$ is a strong starter in $G \setminus H$. Let $\phi : G \rightarrow G/H$ be the canonical homomorphism. The \emph{type} of an element $g \in G$ is $\phi(g)$, and the \emph{type} of a pair $\{x_i,y_i \}$ is defined to be $\{\phi(x_i), \phi(y_i)\}$. Clearly, no pair in $S$ contains two elements of the same type, nor does any pair in $S$ contain an element of type $0$, where $0$ is the identity in $G/H$.
By examining the homomorphic image of a strong starter under $\phi$, it is possible to write down various linear equations. If the resulting system of equations does not have a solution in non-negative integers, then the strong starter cannot exist.

\begin{theorem}\cite{DS80}
\label{t^5strong.thm}
Suppose $t$ is odd, $G$ is an abelian group of order $5t$ and $H$ is a subgroup of order $t$. Then there does not exist a strong frame starter in $G \setminus H$.
\end{theorem}

\begin{proof}
Clearly $G / H \cong \zed_5$. Let $\phi : G \rightarrow \zed_5$ be the canonical homomorphism.
Suppose $S = \{ \{x_i,y_i \} : 1 \leq i \leq 2t\}$ is a frame starter in $G \setminus H$. 
No pair in $S$ contains two elements of the same type, nor does any pair in $S$ contain an element of type $0$.

For $1 \leq i < j \leq 4$, let $a_{\{i,j\}}$ denote the number of pairs in $S$ of type $\{i,j\}$.
Observe that $a_{\{2,3\}} = a_{\{1,4\}} = 0$, because pairs of type $\{2,3\}$ or $\{1,4\}$ would have a sum that is of type $0$ (i.e., the sum would be in $H$). This is of course forbidden.

$G \setminus H$ contains $t$ elements of type $i$, for $1 \leq i \leq 4$. 
Thus the following four equations are obtained:
\begin{eqnarray}
\label{e11}a_{\{1,2\}} + a_{\{1,3\}}    & = & t\\
\label{e12}a_{\{1,2\}} + a_{\{2,4\}}   & = & t\\
\label{e13}a_{\{1,3\}} + a_{\{3,4\}}  & = & t\\
\label{e14}a_{\{2,4\}} + a_{\{3,4\}}   & = & t.
\end{eqnarray}

Now  consider pairs in $S$ that give rise to differences of a specified type. A pair of type $\{1,2\}$ or $\{3,4\}$ gives rise to one difference of type $1$ and one difference of type $4$. Also, a pair of type $\{1,3\}$, $\{2,4\}$ gives rise to one difference of type $2$ and one difference of type $3$. There are $t$ differences of each possible type $i$ (for $1 \leq i \leq 4$), so the following three equations result:
\begin{eqnarray}
\label{e16}a_{\{1,2\}} +  a_{\{3,4\}}     & = & t\\
\label{e17}a_{\{1,3\}} + a_{\{2,4\}}    & = & t.
\end{eqnarray}
Now, from (\ref{e11}) and (\ref{e12}), it follows  that $a_{\{1,3\}}  = a_{\{2,4\}}$.
Then, from (\ref{e16}), it results that $a_{\{1,3\}}  = a_{\{2,4\}} = t/2$. This is impossible, because $t$ is odd.\qed
\end{proof}

\begin{remark}
Theorem \ref{t^5strong.thm} rules out the existence of a strong frame starter in $G \setminus H$ for $G =\zed_{15}$ and $H = \{0,5,10\}$. However, Example \ref{z15.exam} provides orthogonal frame starters in $G \setminus H$.
\end{remark}

I don't know if the following very simple result is  new, but I was unable to find it recorded anywhere. 

\begin{theorem}
\label{newnonexist.thm}
Suppose $G$ is an abelian group of order $4t$ and suppose $H$ is a subgroup of $G$ of order $t$, where $t$ is even and $G / H \cong \zed_4 $. Then there is no strong frame starter in $G \setminus H$.
\end{theorem}

\begin{proof}
We are assuming $G / H \cong \zed_4$. Let $\phi : G \rightarrow \zed_4 $ be the canonical homomorphism.
Suppose $S$ is a frame starter in $G \setminus H$. 
A pair in $S$ cannot contain an element of type $0$ or a difference that is of type $0$, so the possible types of
pairs in $S$ are $\{1,2\}$, $\{2,3\}$  and $\{1,3\}$. Since the starter is strong, there cannot be any pairs of type $\{1,3\}$. However, this means that all pairs in $S$ have type $\{1,2\}$ or $\{2,3\}$. Consequently, there cannot be any pairs in $S$ that have a difference of type $2$, which is a contradiction. \qed 
\end{proof}

\begin{coro}
There is no strong frame starter in $\zed_{4t} \setminus \{0,4,8, \dots 4t-4\}$.
\end{coro}

\begin{remark}
Theorem \ref{newnonexist.thm} does not apply if $G / H \cong \zed_2 \times \zed_2$.
Indeed, Example \ref{z4z4.exam} shows that there is a strong frame starter in $(\zed_4 \times \zed_4) \setminus 
\{(0,0), (0,2), (2,0), (2,2)\}$.
\end{remark}

Here is another new result that has a similar but slightly more intricate proof.

\begin{theorem}
\label{newnonexist2.thm}
Suppose $G$ is an abelian group of order $6t$ and suppose $H$ is a subgroup of $G$ of order $t$. 
Then there is no strong frame starter in $G \setminus H$.
\end{theorem}

\begin{proof}
We have that $G / H \cong \zed_6$. Let $\phi : g \rightarrow \zed_6$ be the canonical homomorphism.
Suppose $S = \{ \{x_i,y_i \} : 1 \leq i \leq 5t/2\}$ is a frame starter in $G \setminus H$. 
No pair in $S$ contains two elements of the same type, nor does any pair in $S$ contain an element of type $0$.

For $1 \leq i < j \leq 5$, let $a_{\{i,j\}}$ denote the number of pairs in $S$ of type $\{i,j\}$.
Observe that $a_{\{2,4\}} = a_{\{1,5\}} = 0$, because pairs of type $\{2,4\}$ or $\{1,5\}$ would have a sum that is of type $0$ (i.e., the sum would be in $H$). This is of course forbidden.

$G \setminus H$ contains $t$ elements of type $i$, for $1 \leq i \leq 5$. 
Thus the following five equations are obtained:
\begin{eqnarray}
\label{e1}a_{\{1,2\}} + a_{\{1,3\}} + a_{\{1,4\}}   & = & t\\
\label{e2}a_{\{1,2\}} + a_{\{2,3\}} + a_{\{2,5\}}   & = & t\\
\label{e3}a_{\{1,3\}} + a_{\{2,3\}} + a_{\{3,4\}}  + a_{\{3,5\}} & = & t\\
\label{e4}a_{\{1,4\}} + a_{\{3,4\}} + a_{\{4,5\}}   & = & t\\
\label{e5}a_{\{2,5\}} + a_{\{3,5\}} + a_{\{4,5\}}   & = & t.
\end{eqnarray}

Now consider pairs in $S$ that give rise to differences of a specified type. A pair of type $\{1,4\}$ or $\{2,5\}$ gives rise to two differences of type $3$. A pair of type $\{1,3\}$ or $\{3,5\}$ gives rise to one difference of type $2$ and one difference of type $4$. Finally, a pair of type $\{1,2\}$, $\{2,3\}$, $\{3,4\}$ or $\{4,5\}$ gives rise to one difference of type $1$ and one difference of type $5$. There are $t$ differences of each possible type $i$ (for $1 \leq i \leq 5$), so the following three equations result:
\begin{eqnarray}
\label{e6}a_{\{1,2\}} + a_{\{2,3\}} + a_{\{3,4\}} + a_{\{4,5\}}    & = & t\\
\label{e7}a_{\{1,3\}} + a_{\{3,5\}}    & = & t\\
\label{e8}a_{\{1,4\}} + a_{\{2,5\}}  & = & t/2.
\end{eqnarray}
From (\ref{e3}) and (\ref{e7}), we see that $a_{\{2,3\}}= a_{\{3,4\}}= 0$.
Now, from (\ref{e6}) and (\ref{e8}), we obtain
\begin{eqnarray}
a_{\{1,2\}} + a_{\{1,4\}} + a_{\{2,5\}}  + a_{\{4,5\}} & = & 3t/2.
\end{eqnarray}
On the other hand, 
from (\ref{e2}) and (\ref{e4}), we obtain
\begin{eqnarray*}
a_{\{1,2\}} + a_{\{1,4\}} + a_{\{2,5\}}  + a_{\{4,5\}} & = & 2t.
\end{eqnarray*}
This is a contradiction.\qed
\end{proof}

The following theorem due to Stinson, which is reported in \cite{DL93}, has a more involved proof. It is somewhat similar to Theorem \ref{newnonexist.thm}. Note that Theorem \ref{t^4orth.thm} only applies to values  $t \equiv2 \bmod 4$, but the conclusion is stronger in that it rules out the existence of orthogonal frame starters in the relevant group. In fact, this the only nonexistence result of which I am aware that specifically applies to orthogonal frame starters.

\begin{theorem}\cite{DL93}
\label{t^4orth.thm}
If $t \equiv 2 \bmod 4$, then there does not exist a pair of orthogonal frame starters  in the group 
$\zed_{4t} \setminus \{0,4,8, \dots 4t-4\}$.
\end{theorem}

\begin{coro} Suppose $G$ is an abelian group of order $4t$ and suppose $H$ is a  subgroup of $G$ of order $t$, where $t \equiv 2 \bmod 4$ is the product of distinct primes. Then there does not exist a pair of orthogonal frame starters in $G \setminus H$. 
\end{coro}

\begin{proof}
We must have $G \cong K \times L$, where $K$ has order $8$ and $L$ has (odd) order $t/2$. Also, $L$ must be cyclic since it is abelian and its order is the product of distinct primes.
We have $K \cong \zed_8, \zed_4 \times \zed_2$ or $\zed_2 \times \zed_2 \times \zed_2$. However, there cannot be any elements of order $2$ in $G \setminus H$. It follows that $K \cong \zed_8$ and  $H = \{0,4\} \times L$,
so $G = \zed_8 \times L$ is cyclic since $L$ has odd order. The conclusion then follows from Theorem \ref{t^4orth.thm}.\qed
\end{proof}

\section{Strong Frame Starters in Cyclic Groups}
\label{cyclic.sec}

Suppose $G$ is a cyclic group of order $g$ and $H$ is a subgroup of $G$ having order $h$, such that $g - h$ is even and $g/h \geq 4$. It is well-known that there is no strong starter in $\zed_9 \setminus \{0\}$. However, modulo this single exceptional case, it seems plausible that there is a strong frame starter in $G \setminus H$ unless its existence  is ruled out by one of the theorems proven in Section \ref{strong.sec}. More precisely, I propose following conjecture.

\begin{conjecture}
\label{main.conj}
Suppose that $G$ is a cyclic group of order $g$ and $H$ is a subgroup of $G$ having order $h$, such that $g - h$ is even and $g/h \geq 4$. Then there is a strong frame starter in $G \setminus H$ if and only if none of the following conditions holds:
\begin{enumerate}
\item $g = 2u$ and $h = 2t$, where $t$ is odd and $g/h \equiv 2,3 \bmod 4$,
\item $h$ is odd and $g = 5h$,
\item $g = 4h$ or $g = 6h$,
\item $h=1$ and $g = 9$.
\end{enumerate}
\end{conjecture}

There is at least some fairly convincing empirical evidence for the correctness of Conjecture \ref{main.conj}. I have verified that there is a strong frame starter in $G \setminus H$, provided that the conditions stated in the conjecture do not hold, for all $g \leq 100$. The relevant strong frame starters were all constructed using the hill-climbing algorithm described in \cite{DS92a}. The pairs $(h,g)$, $g \leq 100$, for which strong starters in $\zed_g \setminus \zed_h$ exist are summarized in Table \ref{table1}.

Most of these strong frame starters were constructed extremely quickly. (I used a simple implementation in Maple to find them.) The only cases that required more time or multiple trials were those with $g = 5h$. Since these cases were a bit more difficult, I include examples of these strong starters in the Appendix. 

\begin{table}[t]
\caption{Parameters for which strong starters exist in $\zed_g \setminus \zed_h$, $g \leq 100$}
\label{table1}
\begin{center}
\begin{tabular}{r|l}
\multicolumn{1}{c|}{$h$} & \multicolumn{1}{c}{$g$} \\ \hline 
$1$ & $7, 11, 13,15, \dots , 99$\\
$2$ & $10, 16, 18, 24, 26, 32, 34, 40, 42, 48, 50, 
 56, 58, 64, 66, 72, 74, 80, 82, 88, 90, 96,98$\\
$3$ & $21, 27, \dots, 99$\\
$4$ & $20, 28, 32, \dots , 100$\\
$5$ & $35, 45, 55, 65, 75, 85, 95$\\
$6$ & $30, 48, 54, 72, 78, 96$\\
$7$ & $49, 63, 77, 91$\\
$8$ & $40, 64, 72, 80, 88, 96$\\
$9$ & $63,81,99$\\
$10$ & $50,80,90$\\
$11$ & $77,99$\\
$12$ & $60,84,96$ \\
$13$ & $91$ \\
$14$ & $70$ \\
$16$ & $80$\\
$18$ & $90$ \\
$20$ & $100$
\end{tabular}
\end{center}
\end{table}

The special case of Conjecture \ref{main.conj} with $h=1$ is contained in the following conjecture due to Horton \cite{Ho89}, which is not restricted to cyclic groups.  Although Horton's conjecture dates from the late 1980's, it still remains far from being solved.

\begin{conjecture}[Horton]
Suppose that $G$ is an abelian group of odd order $g$ where $g \geq 3$. Then there is a strong starter in $G$ if and only if $G \neq \zed_3, \zed_5, \zed_9$ or $\zed_3 \times \zed_3$.
\end{conjecture} 

Also, note that the special cases of Conjecture \ref{main.conj} with $h=1,2$ are also mentioned in the 1992 survey article \cite{DS92} as open problems.

\section{Discussion and Conclusion}
\label{discuss.sec}

in 1981, Jeff Dinitz and I developed a hill-climbing algorithm to find strong starters in cyclic groups; see \cite{DS81a}. This was notable as being the first successful use of a hill-climbing algorithm to find a nontrivial combinatorial structure. At the time, we did not even know what a hill-climbing algorithm was, so the title of \cite{DS81a} instead refers to a ``fast'' algorithm. The algorithm is a randomized heuristic algorithm and there is no proof that it will actually succeed in finding a given strong starter. However, it has turned out to be very successful in practice. This hill-climbing algorithm can easily be modified to search for strong frame starters, as described in \cite{DS92a}. 

I find it very intriguing that it is seemingly very easy in practice to construct strong frame starters  in any desired $G \setminus H$ (assuming that the existence of such a frame starter is not ruled out by a non-existence theorem).
However, more than forty years after the development of this hill-climbing algorithm, there is still no proof that all ``possible'' strong starters or strong frame starters exist. Perhaps this is due in part to a lack of attention paid to investigating this problem in recent years. Hopefully by highlighting this problem, it may lead to future progress.

Another observation in the course of writing this paper that I find interesting is that there turned out be nonexistence results for strong frame starters that were not previously recorded.  I suppose this is because the corresponding frames were constructed by other techniques, so the strong frame starters were not required for this specific purpose. In any event, there turned out to be interesting mathematics involved in taking a new look at an old problem.

As far as future research is concerned, there is still much to be said about existence or nonexistence of strong frame starters in non-cyclic groups. A couple of examples I have mentioned in this paper provide hints that there are interesting lines of investigation to be addressed. 
One such example is the existence of orthogonal frames starters in $\zed_{15} \setminus \{0,5,10\}$, 
when a strong frame starter does not exist. Another is the existence of a strong frame starter in $\zed_{16} \setminus \{(0,0), (0,2), (2,0), (2,2) \}$ when there is no strong frame starter in $\zed_{16} \setminus \{0,4,8,12\}$.

Finally, I should also mention that there are also interesting open problems regarding \emph{skew} frame starters, a topic that I have not mentioned in this paper.

\appendix

\section{Some Strong Frame Starters in $\zed_{5t} \setminus \{0,5,10, \dots 5t-5\}$}

Table \ref{tablexx} provides examples of strong frame starters in $\zed_{5t} \setminus \{0,5,10, \dots 5t-5\}$ for $t$ even, $4 \leq t \leq 20$. Note that the case $t =2$ was done in Example \ref{zed10}. Also, the case $t = 14$ can  be obtained from Remark \ref{5h.rem}.

  \begin{table}
   \caption{Strong frame starters in $\zed_{5t} \setminus \{0,5,10, \dots 5t-5\}$}
   \label{tablexx}
\begin{center}
\begin{tabular}{r|l}
  $t$ & \multicolumn{1}{c}{strong frame starter}  \\ \hline
 4 &   $\{ 16, 17\},  \{ 1, 3\}, \{ 9, 12\},\{ 2, 6\},\{ 8, 14\}, \{ 11, 18\}, \{ 19, 7\}, \{ 4, 13\}$
\\ \hline
  6  &  $\{  3, 4 \}, \{26, 28\}, \{9, 12\}, \{17, 21\}, \{ 8, 14\}, \{16, 23\}, \{19, 27\}, \{2, 11\}$\\
& $\{18, 29\}, \{1, 13\},\{24, 7\}, \{22, 6\}$
\\ \hline
  8  & $ \{  28, 29 \}, \{17, 19\}, \{8, 11\}, \{39, 3\}, \{ 16, 22\}, \{ 2, 9\}, \{24, 32\}, \{27, 36\}$\\
& $\{1, 12\}, \{26, 38\},\{34, 7\}, \{4, 18\}, \{21, 37\}, \{6, 23\},\{13, 31\}, \{14, 33\}$\\ \hline
 10  &  $\{  26,27 \}, \{7,9\}, \{28,31\}, \{34,38\}, \{ 11,17\}, \{ 16,23\}, \{44,2\}, \{32,41\}$\\
& $\{3,14\}, \{36,48\},\{49,12\}, \{42,6\}, \{13,29\}, \{1,18\},\{19,37\}, \{39,8\}$\\
& $\{33,4\}, \{21,43\},\{24,47\}, \{22,46\}$\\ \hline
 12  &  $\{  46, 47 \}, \{41, 43\}, \{58, 1\}, \{2, 6\}, \{ 21, 27\}, \{ 52, 59\}, \{23, 31\}, \{17, 26\}$\\
& $\{56, 7\}, \{32, 44\},\{9, 22\}, \{ 19, 33\}, \{38, 54\}, \{11, 28\},\{24, 42\}, \{57, 16\}$\\
& $\{18, 39\}, \{51, 13\},\{49, 12\}, \{29, 53\}, \{48, 14\}, \{37, 4\},\{8, 36\}, \{34, 3\}$
\\ \hline
 14  &  $\{  43, 44 \}, \{56, 58\}, \{3, 6\}, \{19, 23\}, \{ 11, 17\}, \{  42, 49\}, \{28, 36\}, \{12, 21\}$\\
& $\{68, 9\}, \{2, 14\},\{53, 66\}, \{ 24, 38\}, \{51, 67\}, \{52, 69\},\{29, 47\}, \{27, 46\}$\\
& $\{18, 39\}, \{41, 63\},\{8, 31\}, \{62, 16\}, \{48, 4\}, \{7, 34\},\{33, 61\}, \{54, 13\}$\\ 
& $\{26, 57\}, \{ 32, 64\},\{59, 22\}, \{37, 1\}$\\ \hline
 16  &  $\{  36, 37 \}, \{51, 53\}, \{28, 31\}, \{57, 61\}, \{ 78, 4\}, \{  7, 14\}, \{33, 41\}, \{64, 73\}$\\
& $\{58, 69\}, \{27, 39\},\{19, 32\}, \{ 54, 68\}, \{56, 72\}, \{66, 3\},\{74, 12\}, \{77, 16\}$\\
& $\{1, 22\}, \{ 67, 9\},\{29, 52\}, \{ 2, 26\}, \{21, 47\}, \{11, 38\},\{34, 62\}, \{17, 46\}$\\ 
& $\{18, 49\}, \{ 71, 23\},\{43, 76\}, \{59, 13\}, \{8, 44\}, \{ 42, 79\},\{48, 6\}, \{24, 63\}$\\ \hline
 18  &  $\{  58, 59 \}, \{12, 14\}, \{89, 2\}, \{22, 26\}, \{ 68, 74\}, \{  66, 73\}, \{49, 57\}, \{42, 51\}$\\
& $\{61, 72\}, \{36, 48\},\{69, 82\}, \{ 84, 8\}, \{78, 4\}, \{86, 13\},\{88, 16\}, \{19, 38\}$\\
& $\{46, 67\}, \{  32, 54\},\{53, 76\}, \{ 39, 63\}, \{18, 44\}, \{7, 34\},\{9, 37\}, \{62, 1\}$\\ 
& $\{21, 52\}, \{ 11, 43\},\{23, 56\}, \{87, 31\}, \{83, 29\}, \{ 77, 24\},\{3, 41\}, \{79, 28\}$\\ 
& $\{6, 47\}, \{ 81, 33\},\{64, 17\}, \{27, 71\}$\\ \hline
 20  &  $\{  76, 77 \}, \{7, 9\}, \{94, 97\}, \{79, 83\}, \{ 66, 72\}, \{  37, 44\}, \{84, 92\}, \{89, 98\}$\\
& $\{48, 59\}, \{36, 48\},\{58, 71\}, \{ 54, 68\}, \{6, 22\}, \{2, 19\},\{93, 11\}, \{99, 18\}$\\
& $\{13, 34\}, \{  42, 64\},\{23, 46\}, \{ 27, 51\}, \{31, 57\}, \{16, 43\},\{33, 61\}, \{52, 81\}$\\ 
& $\{36, 67\}, \{ 17, 49\},\{88, 21\}, \{4, 38\}, \{78, 14\}, \{ 26, 63\},\{3, 41\}, \{69, 8\}$\\ 
& $\{91, 32\}, \{ 86, 28\},\{53, 96\}, \{12, 56\}, \{1, 47\}, \{ 82, 29\},\{39, 87\}, \{24, 73\}$
\end{tabular}
\end{center}
\end{table}

\end{document}